\documentclass[twoside,11pt]{article}
\usepackage{amsmath,amssymb,amsthm,amsfonts,mathrsfs,wasysym,latexsym,times,lineno, subfigure,color,booktabs}
\usepackage{amsmath,amsfonts}
\usepackage{multirow}
\usepackage{array}
\usepackage{amsthm}
\usepackage{graphicx}
\usepackage{color}
\usepackage{multicol}
\usepackage{float}

\newtheorem{lemma}{Lemma}[section]

\newtheorem{theorem}{Theorem}[section]

  \date{}
\title{Exact Phase Transitions of Model RB with Slower-Growing Domains}

\author{Jun Liu$^{1}$, \ \ Ke Xu$^{2}$,\ \ Guangyan Zhou$^{3,*}$\\
\small $^1$Department of Mathematics, Taiyuan University of Technology, Taiyuan, China\\
\small $^2$ State Key Laboratory of Software Development Environments, Beihang University, China\\[-0.8ex]
\small $^3$Department of Mathematics, Beijing Technology and Business University,  China\\
 }
\begin{document}
\maketitle

\begin{abstract}
The second moment method has always been an effective tool to lower bound the satisfiability threshold of many random constraint satisfaction problems.
However, the calculation is usually hard to carry out and as a result, only some loose results can be obtained. In this paper, based on a delicate analysis which fully exploit the power of the second moment method, we prove that random RB instances can exhibit exact phase transition under more relaxed conditions, especially slower-growing domain size. These results are the best by using the second moment method, and new tools should be introduced for any better results.

\end{abstract}
\footnotetext[1]{Corresponding author. jvnliu@gmail.com, kexu@nlsde.buaa.edu.cn, gyzhou76@btbu.edu.cn.

 This research was supported
by National Natural Science Fund of China (Grant No.61702019). }
\section{Introduction}
The study of random constraint satisfaction problems (CSPs) has received tremendous ideas from combinatorics, computer science and statistical physics. Random CSPs contain a large set of variables which interact through a large set of constraints, where each variable ranges over a domain and a configuration (solution) to all the variables should satisfy all of the constraints. A fundamental question in the study of random CSPs is the average-case computational complexity of solving ensembles of CSPs. Great amount of theoretical and algorithmic work has been devoted to establish and locate the satisfiability threshold, and studies show that complexity attains the maximum at the SAT-UNSAT transition.

Many of the studied CSP models (such as random $k$-SAT, graph coloring) have fixed domain size, constraint length, and the number of constraints is linear  compared with the number of variables. In recent years, a lot of attention has been paid to the study of CSPs with growing domains or constraint length (\cite{xu2000,Frieze2005,FS2011,FSX2012}). For instance, random
$k$-SAT model with moderately growing $k$ has been proved to have a tight
threshold under the condition that $k-\log_2n\rightarrow\infty$ \cite{Frieze2005}, which can be relaxed as $k\geq(\frac12+\theta)\log_2n$ for any fixed $\theta>0$
\cite{liu2012}. Model RB is a standard prototype CSP model with growing domains revised from the Model B \cite{smith1996}. The proposal of this model is to overcome the trivial asymptotic insolubility of Model B. Model RB has been proved to have sharp SAT-UNSAT phase transition and exact  threshold points, and can generate hard instances in the phase transition region \cite{xu19,xu2006}. In addition, it was proved that model RB has a clustering transition but no condensation transition\cite{zhao2012,xuwei2015}. Moreover, randomly generated forced  RB instances with one hidden solution are proved to have both similar distribution of solutions and hardness property with unforced satisfiable RB instances \cite{xu2007}. Benchmarks based on the exact phase transitions of model RB (e.g. BHOSLIB) have been extensively used in algorithm research and in various algorithm competitions (e.g. CSP, SAT , MVC and MaxSAT), and the results confirm the hardness of these benchmarks (online at www.nlsde.buaa.edu.cn/~kexu/).

In this paper, we revisit the satisfiability threshold of model RB, and obtain new results based on techniques build upon the second moment method.  Precisely, we show that model RB can exhibit sharp phase transitions under more relaxed restrictions, especially a slower-growing domain size, and such results will greatly simplify the generation of satisfiable benchmarks, thus help the evaluation of various CSP algorithms. Technically, we get a more precise estimate of the upper bound of ${\mathbb{E}[X^{2}]}/{\mathbb{E}[X]^{2}}$, which has the form $\sum_{\omega=0}^{n}{n \choose \omega}p_{n}^{\omega}(1-p_{n})^{n-\omega}f_{n}(\omega/n)^{m}$, by giving a delicate partition of integers in $[0,n]$ and then analysing the monotonous or asymptotic behaviors in each interval. It is worth mention that our results are the best under the second moment method, since it fails to give nontrivial results once the conditions are further relaxed (see Claim 5.2), thus new tools other than the second moment method will be needed to entail better results.


\section{Background and Motivations}

Generally, a CSP is a triple $\langle X, D, C \rangle$, where $X = \{x_{1},x_{2},\ldots,x_{n}\}$ is a set of variables; $D = \{D_{1},D_{2},\cdots,D_{n}\}$ is a set of domains,
where $D_{i}$ is the domain of variable $x_{i}$, for $i=1,2,\cdots,n$; $C = \{C_{1},C_{2},\cdots,C_{t}\}$ is a set of constraints. For each $1\leq i \leq t$, constraint $C_{i}$ is a pair $\langle X_{i},R_{i} \rangle$,
where $X_{i} = \{x_{i_1},x_{i_2},\cdots,x_{i_{n_{i}}}\}\subseteq X$ is a set of variables and $R_{i}\subseteq D_{i_1}\times D_{i_2}\times \cdots \times D_{i_{n_{i}}}$ is a set of tuples of values.
A constraint $C_{i}$ is satisfied if the tuple of values assigned to the variables $x_{i1},x_{i2},\cdots,x_{in_{i}}$ is in the relation $R_{i}$. A solution is an assignment to all variables where all the constraints can be satisfied.
A CSP that has a solution is called satisfiable; otherwise it is unsatisfiable.

In Model RB \cite{xu2000}, the domain size $|D_{i}| = d$  for each $1\leq i\leq n$ ($d = n^{\alpha}$, where $\alpha > 0$ is a constant); the constraint length $n_{i} = k$ ($k\geq 2$ is a positive integer); $|R_{i}| = (1-p)d^{k}$ ($0 < p < 1$ is a constant); $t = r n\ln d$ ($r > 0$ is a constant).\footnote{We are very grateful to Prof. Donald E. Knuth for this suggestion on using $t = r n \ln d$ instead of $t = r n \ln n$ to denote the number of constraints, which makes the result look nicer.}

Given parameters $(n,\alpha,r,k,p)$, a random RB instance can be  generated  as follows \cite{xu2000}.

\emph{Step 1}. \emph{Select with repetition} $t=r n\ln d$ \emph{random constraints}. \emph{Each random constraint is formed by selecting without repetition} $k$ ($k \geq 2$) \emph{of} $n$ \emph{variables}.

\emph{Step 2}. \emph{For each constraint we uniformly select without repetition} $q=p\times d^{k}$ \emph{incompatible tuples of values}.

The following results \cite{xu2000} of Model RB establish the exact satisfiability transition points, where we use $X$ to denote the number of solutions in a given instance.

\begin{theorem}\label{th21}
Let $r_{cr} = \frac{1}{\ln\tau}$, where $\tau=\frac{1}{1-p}$.
If $k\alpha>1$ and $k\geq \tau$, then
\begin{align}
\lim_{n\rightarrow \infty}\mathrm{Pr}[X>0]=1\ \text{ if }\ r < r_{cr},\\
\lim_{n\rightarrow \infty}\mathrm{Pr}[X>0]=0\ \text{ if }\ r > r_{cr}.
\end{align}
\end{theorem}

\begin{theorem}\label{th22}
 Let $p_{cr}=1-\frac{1}{\tau}$, where $\tau=e^{ \frac{1}{r} }$.
If $k\alpha>1$ and $k\geq \tau$, then
\begin{align}
\lim_{n\rightarrow \infty}\mathrm{Pr}[X>0]=1\ \text{ if }\ p < p_{cr},\\
\lim_{n\rightarrow \infty}\mathrm{Pr}[X>0]=0\ \text{ if }\ p > p_{cr}.
\end{align}
\end{theorem}

In \cite{c2}, Achlioptas et al. proved that standard random CSP models with fixed domains suffer from trivial asymptotic unsatisfiability.
An important feature of Model RB is that the domain size $d=n^{\alpha}$ ($\alpha >0$ is a constant) grows with the number of variables $n$, which can be used to model
some practical problems such as the $n$-queen problem, the Latin square problem and the maximum acyclic subgraph problem \cite{e1} .
From the theorems above, we can see that Model RB not only avoids trivial asymptotic behavior but also has phase transitions whose threshold
values can be located exactly. In \cite{c7}, Frieze and Molloy determined how fast the domain size has to grow in order to exhibit phase transition behavior.
Along this line of research, an interesting open problem is to determine how fast the domain size has to grow in order to exhibit phase transitions
while still having exact threshold values, which is the main motivation of this paper.

\section{Main Results}
In this paper, we fully exploit the power of the second moment method on Model RB and show that the requirements of $\alpha>\frac{1}{k}$ and $k\geq\tau$ in Theorems 2.1 and 2.2 can be relaxed as $\alpha>\frac{1}{2k-1}$ and $k\geq \frac{\tau\ln \tau}{\tau-1}$.

\begin{theorem}\label{th31}
Let $r_{cr} = \frac{1}{\ln\tau}$, where $\tau =\frac{1}{1-p}$.
If $(2k-1)\alpha>1$ and $k \geq \frac{\tau\ln \tau}{\tau -1}$, then
\begin{align}
\lim_{n\rightarrow \infty}\mathrm{Pr}[X>0]=1\ \text{ if }\ r < r_{cr},\label{a}\\
\lim_{n\rightarrow \infty}\mathrm{Pr}[X>0]=0\ \text{ if }\ r > r_{cr}.\label{3}
\end{align}
\end{theorem}

\begin{theorem}\label{th32}
Let $p_{cr}=1-\frac{1}{\tau}$, where $\tau =e^{ \frac{1}{r} }$.
If $(2k-1)\alpha>1$ and $k \geq \frac{\tau\ln \tau}{\tau -1}$, then
\begin{align}
\lim_{n\rightarrow \infty}\mathrm{Pr}[X>0]=1\ \text{ if }\ p < p_{cr},\label{b}\\
\lim_{n\rightarrow \infty}\mathrm{Pr}[X>0]=0\ \text{ if }\ p > p_{cr}.\label{4}
\end{align}
\end{theorem}

Since $\tau > 1$, then $0<\frac{\tau \ln \tau}{\tau-1}<\tau$.
$\frac{\tau \ln \tau}{\tau-1}\sim \ln\tau=o(\tau)$ as $\tau\rightarrow \infty$.

Let $\zeta(z) = \frac{z \ln z}{z - 1}, z > 1$. Since $\zeta^{'}(z) = \frac{(z - 1)-\ln z}{(z - 1)^{2}}>0, z > 1$, then
$\zeta(z)$ increases monotonously in $(1, +\infty)$.
It is easy to see that $\lim_{z \rightarrow 1^{+}} \zeta(z) = 1$ and $\lim_{z \rightarrow +\infty} \zeta(z) = +\infty$.
Thus, for any fixed $k\geq 2$, the equation $\zeta(z) = k, z > 1$ has an unique solution $\tau_{k} > k$.
$\tau_{k} \sim e^{k}$ as $k \rightarrow \infty$.
Then, $\tau\leq\tau_{k}$, i.e., $p\leq 1-\frac{1}{\tau_{k}}$ in Theorem 3.1, and $r \geq \frac{1}{ \ln \tau_{k} }$ in Theorem 3.2,
compared with $\tau \leq k$ in Theorems 2.1 and 2.2, i.e., $p\leq 1-\frac{1}{k}$ in Theorem 2.1, and $r \geq \frac{1}{ \ln k }$ in Theorem 2.2.

Thus, we have the table
\begin{center}
\begin{tabular}{lllll}
$k$   &$p$ in Th. 3.1                &$p$ in Th. 2.1                &$r$ in Th. 3.2                &$r$ in Th. 2.2\\
\hline
2     &$\leq 0.79681\cdots$          &$\leq 0.50000\cdots$          &$\geq 0.62750\cdots$          &$\geq 1.44269\cdots$\\
3     &$\leq 0.94047\cdots$          &$\leq 0.66666\cdots$          &$\geq 0.35442\cdots$          &$\geq 0.91023\cdots$\\
4     &$\leq 0.98017\cdots$          &$\leq 0.75000\cdots$          &$\geq 0.25505\cdots$          &$\geq 0.72134\cdots$\\
5     &$\leq 0.99302\cdots$          &$\leq 0.80000\cdots$          &$\geq 0.20140\cdots$          &$\geq 0.62133\cdots$\\
\end{tabular}
\end{center}

Our results are obtained by  fully exploiting the power of the second moment method in
Model RB in that, $(2k-1)\alpha >1$ and $k \geq\frac{\tau\ln \tau}{\tau -1}$ in Theorems 3.1 and 3.2 are necessary conditions for the second moment method
to give nontrivial results (See Claims 5.2 and
5.3).

\section{Some Lemmas}
Xu and Li  \cite{xu2000} have proved that if $k\alpha >1$ and $k\geq \tau$, the conclusions of Theorems 3.1 and 3.2 are true. By using the same analysis with \cite{xu2000} and Lemma 4.8, it is not hard to show
that if $k\alpha >1$ and $k\geq \frac{\tau\ln\tau}{\tau-1}$, then Theorems 3.1 and 3.2 hold.
Thus, we only need to consider the case of $k\alpha \leq 1$ in the following proof.

In the following, we tacitly assume that $(2k-1)\alpha>1$, $k \geq \frac{\tau\ln \tau}{\tau -1}$, and $r < r_{cr}$  unless otherwise specified.

As before, let $X$ be the number of solutions of a random  RB instance and let $m = n \ln d$. From \cite{xu2000}, we have
\begin{align}
\mathbb{E}[X] = d^{n}(1-p)^{ r m },\label{5}
\end{align}
\begin{align}
\mathbb{E}[X^{2}] = \sum_{S=0}^{n} d^{n} {n \choose S} (d-1)^{n-S} \left(\frac{{d^{k}-1 \choose q}}{{d^{k} \choose q}}\frac{{S \choose k}}{{n \choose k}}+\frac{{d^{k}-2 \choose q}}{{d^{k} \choose q}}\left(1-\frac{{S \choose k}}{{n \choose k}}\right)\right)^{ r m }.\label{30}
\end{align}

Simple calculation yields
\begin{align}
\frac{{d^{k}-1 \choose q}}{{d^{k} \choose q}}=1-p,\ \ \text{and}\ \frac{{d^{k}-2 \choose q}}{{d^{k} \choose q}}=(1-p)^{2} - p(1-p)\frac{ d^{-k} }{ 1 - d^{-k} }.\label{7}
\end{align}

Let $g(s) = -k(k-1)(1-s)s^{k-1}/2$, where $s = S/n$. Then
\begin{align}
\nonumber
\frac{{S \choose k}}{{n \choose k}} = & \frac{ s \left( s - \frac{1}{n} \right) \left( s - \frac{2}{n} \right) \cdots \left( s - \frac{k-1}{n} \right) }{ 1 \left( 1 - \frac{1}{n} \right) \left( 1 - \frac{2}{n} \right) \cdots \left( 1 - \frac{k-1}{n} \right) } \\
\nonumber
= & \frac{ s^{k} - \frac{k(k-1)}{2n} s^{k-1} + O\left( \frac{1}{n^{2}} \right) }{ 1 - \frac{k(k-1)}{2n} + O\left( \frac{1}{n^{2}} \right) } \\
= & s^{k}+\frac{g(s)}{n}+o\left( \frac{1}{m} \right),\label{8}
\end{align}
where $o\left( \frac{1}{m} \right)$ is independent of $s$ (i.e., independent of $S$).

Let $f(s)=1+\frac{p}{1-p}\frac{ s^{k} - d^{-k} }{ 1 - d^{-k} }$, $s\in [0, 1]$.
For $S = 1, 2, \cdots, n$, let $\Phi (S)=\text{B}(S)\text{W}(S)$,
where $\text{B}(S)={n \choose S}  \left(\frac{1}{d}\right)^{S}  \left(1-\frac{1}{d}\right)^{n-S}$, and $\text{W}(S)=f(S/n)^{ r m }$.

Let $\eta_{1}=\frac{1}{d}+\frac{ \lambda }{n^{1-(k-1)\alpha}\ln d}$ ($\lambda > 0$ is an arbitrarily given constant),
and let $\eta_{2}, \eta_{3}$, and $\mu$ be constants such that
$0<\eta_{2}<\eta_{3}<1$,
$\alpha_{0}/\alpha - \mu \eta_{2}^{k-1} > 0$ with $\alpha_{0} = \frac{(2k-1)\alpha-1}{2(k-1)}$, $\mu > \frac{ kpr }{1-p}$,
and $r\ln (1-p) + \eta_{3} > 0$.

In order to prove ($\ref{a}$), we need some lemmas as follows.

\begin{lemma}\label{lemma41}
Fix any $\alpha > 0$, then
\begin{align}
\nonumber
\frac{\mathbb{E}[X^{2}]}{\mathbb{E}[X]^{2}} = &\Big(1 + o(1) \Big) \sum_{S=0}^{n} \text{B}(S)\left( f(s) + \frac{ p g(s) }{ (1-p) \left(1-d^{-k}\right) n} \right)^{ r m }\\
\nonumber
\leq                                          &\Big(1 + o(1) \Big)\sum_{S=0}^{n}\Phi(S).
\end{align}
\end{lemma}
\begin{proof}
Substituting ($\ref{7}$), ($\ref{8}$) into ($\ref{30}$), and note that $o\left( \frac{1}{m} \right)$ in (\ref{8}) is independent of $s$, then
\begin{align}
\nonumber
     &\frac{\mathbb{E}[X^{2}]}{\mathbb{E}[X]^{2}}=\sum_{S=0}^{n} \text{B}(S) \left(f(s) +  \frac{ p g(s) }{ (1-p) \left(1-d^{-k}\right) n} + o\left( \frac{1}{m} \right)\right)^{ r m }\\
\nonumber
=    &\Big(1 + o(1) \Big) \sum_{S=0}^{n} \text{B}(S)\left( f(s) + \frac{ p g(s) }{ (1-p) \left(1-d^{-k}\right) n} \right)^{ r m } \leq \Big(1 + o(1) \Big)\sum_{S=0}^{n}\Phi(S).
\end{align}
Thus, the proof of Lemma 4.1 is finished.
\end{proof}

\begin{lemma}\label{lemma42}
Fix any $\alpha > 0$, if $k\alpha\leq 1$, then $\text{W}(n\eta_{1}) = \exp\left\{\frac{ k\lambda p r }{1-p}\right\} + o(1)$.
\end{lemma}
\begin{proof}Note that $\eta_{1}^{k} - d^{-k} = \frac{ k\lambda }{m} + o\left( \frac{1}{m} \right)$, then
$f(\eta_{1}) = 1 + \frac{ k\lambda p }{ (1 - p)m } + o\left( \frac{1}{m} \right)$.
Thus, $\text{W}(n\eta_{1})=\left(1+\frac{ k\lambda p }{ (1-p)m } + o\left( \frac{1}{m} \right)\right)^{ r m } = \exp\left\{\frac{ k\lambda p r }{1-p}\right\} + o(1)$,
and we finish the proof of Lemma 4.2.
\end{proof}

\begin{lemma}\label{lemma43}
 $\sum_{S=0}^{n\eta_{1}}\Phi (S) \leq \exp\left\{\frac{ k\lambda p r }{1-p}\right\} + o(1)$.
\end{lemma}
\begin{proof}
\ Note that $\Phi (S)=\text{B}(S)\text{W}(S)$, and $\text{W}(S)$ is monotonically increasing. By Lemma 4.2, we have
\begin{align}
\nonumber
\sum_{S=0}^{n\eta_{1}}\Phi (S) \leq & \text{W}(n\eta_{1}) \sum_{S=0}^{n\eta_{1}} \text{B}(S) \leq \text{W}(n\eta_{1}) \\
\nonumber
= & \exp\left\{\frac{ k\lambda p r }{1-p}\right\} + o(1),
\end{align}
which completes the proof.
\end{proof}

\begin{lemma}\label{lemma44}
Let $\Phi_{c}$ be a function on $[0, n]$ as
\begin{align}
\Phi_{c}(z) = \frac{ \Gamma(n + 1) }{ \Gamma(z + 1)\Gamma(n - z + 1) } \left( \frac{1}{d} \right)^{z}  \left( 1 - \frac{1}{d} \right)^{n-z} f\left( \frac{z}{n} \right)^{ r m }, \label{c1}
\end{align}
where $\Gamma$ is the Gamma Function. Then $\Phi_{c}(\omega) = \Phi(\omega)$, $\omega = 1, 2, \cdots, n$, and $\Phi_{c}^{'}(z) < 0$ if $z \in (n\eta_{1}, n\eta_{2})$.

\end{lemma}

\begin{proof}\ Let $s = z/n$, and take the logarithm of both sides of (\ref{c1}), then the derivative works out to be
\begin{align}
\left( \ln\Phi_{c} \right)^{'} = \frac{ \Phi_{c}^{'} }{ \Phi_{c} } = \text{A}(z) - \ln( d - 1 ) + \frac{ k p r }{(1-p)(1-d^{-k})f(s)} s^{k-1} \ln d, \label{c4}
\end{align}
where $\text{A}(z) = -\frac{ \Gamma^{'}( z + 1 ) }{ \Gamma( z + 1 ) } +\frac{ \Gamma^{'}( n - z + 1 ) }{ \Gamma( n - z + 1 ) }$.

For any positive number $z$,
\begin{align}
-\frac{ \Gamma^{'}(z) }{ \Gamma(z) } = \frac{1}{z} + \gamma + \sum_{i = 1}^{\infty} \left( \frac{1}{ i + z} - \frac{1}{i} \right), \label{c5}
\end{align}
where $\gamma$ is Euler's constant.
For integer $\omega = 0, 1, 2, \ldots$, we have
\begin{align}
-\frac{ \Gamma^{'}(\omega + 1) }{ \Gamma(\omega + 1) } = \gamma - \sum_{i = 1}^{\omega} \frac{1}{i}, \label{c2}
\end{align}
and we know from \cite{a15} that
\begin{align}
-\ln\omega -\frac{1}{2\omega} < \gamma -\sum_{i = 1}^{\omega}\frac{1}{i} < -\ln\omega -\frac{1}{2(\omega+1)}. \label{c3}
\end{align}

By (\ref{c2}) and (\ref{c3}),
\begin{align}
\ln\left( \frac{n}{\omega} - 1 \right) + \text{R}(\omega) < \text{A}(\omega) < \ln\left( \frac{n}{\omega} - 1 \right) + \text{R}(\omega + 1), \label{c6}
\end{align}
where $\text{R}(\omega) = \frac{1}{2} \left( \frac{1}{n - \omega + 1} - \frac{1}{\omega} \right)$, and $\omega = 1, 2, \ldots, n$.

Note that $\mu > \frac{ k p r }{1 - p}$, then for all $z \in (n\eta_{1}, n\eta_{2})$, by (\ref{c4}),
\begin{align}
\left( \ln\Phi_{c} \right)^{'} < \text{A}(z) - \ln (d-1) + \mu s^{k-1}\ln d.\label{13}
\end{align}

Let $\eta_{m} = \frac{1}{n^{\alpha_{1}}}$, where $\alpha_{1}=\frac{1-\alpha}{2(k-1)}$.
Note that $0<\alpha_{1}<\alpha$ and $(k-1)\alpha_{1}>1-k\alpha$.

By (\ref{c5}), $\text{A}(x)$ is decreasing in $(0, \infty)$, thus we can consider $n\eta_{1}$, $n\eta_{m}$ as integers in the following.

(1) If $\eta_{1}\leq s\leq \eta_{m}$, by (\ref{c6}) and ($\ref{13}$), we have
\begin{align}
\nonumber
                    &\left( \ln\Phi_{c} \right)^{'} < \ln\left(\frac{1}{\eta_{1}}-1\right)-\ln (d-1)+\frac{ \mu \ln d }{ n^{(k-1)\alpha_{1}} }\\
\nonumber
=                   &\ln\frac{1-\eta_{1}}{ 1 - \eta_{1}+\frac{ \lambda }{ n^{1-k\alpha} \ln d } }+\frac{ \mu \ln d }{n^{(k-1)\alpha_{1}}} < -\ln \left( 1 + \frac{ \lambda }{ n^{1-k\alpha}\ln d } \right)+\frac{ \mu \ln d }{ n^{(k-1)\alpha_{1}} }\\
\nonumber
=                   &- \frac{ \lambda + o(1) }{ n^{ 1 - k\alpha } \ln d } + \frac{ \mu \ln d }{n^{(k-1)\alpha_{1}}} = - \frac{ \lambda + o(1) }{ n^{1-k\alpha} \ln d} < 0.
\end{align}

(2) If $\eta_{m} \leq s\leq \eta_{2}$, note that $\alpha_{0} = \alpha-\alpha_{1}$, by (\ref{c6}) and ($\ref{13}$), we have
\begin{align}
\nonumber
\left( \ln\Phi_{c} \right)^{'} < &\ln\left( n^{\alpha_{1}}-1 \right) + \text{R}(n) -\ln (d-1)+\mu \eta_{2}^{k-1} \ln d\\
\nonumber
=                                &- \left( \alpha_{0}/\alpha - \mu \eta_{2}^{k-1} + o(1) \right) \ln d < 0.
\end{align}

This finishes the proof of Lemma \ref{lemma44}.
\end{proof}

\begin{lemma}\label{lemma45}
\ ${n \choose ns}=o(\psi(s)^{n})$, \emph{provided that} $ns(1-s)\rightarrow\infty$, \emph{where} $\psi (s)=1/(s^{s}(1-s)^{1-s})$.

\end{lemma}

\begin{proof}\ Note that $ns(1-s)\rightarrow \infty$ is equivalent to $ns\rightarrow\infty$ and $n(1-s)\rightarrow\infty$,
then by Stirling's formula, we have
\begin{align}
\nonumber
{n \choose ns}=        &\frac{1+o(1)}{\sqrt{2\pi ns(1-s)}}\frac{n^{n}}{(ns)^{ns}(n-ns)^{n-ns}}=o(\psi(s)^{n}).
\end{align}
Thus, the proof is finished.
\end{proof}

\begin{lemma}\label{lemma46}
\ $\Phi(n\eta_{1})= o\left( \exp\left\{ - \left( \frac{ \lambda^{ 2 } }{2} + o(1) \right)\frac{n^{(2k-1)\alpha-1}}{(\ln d)^{2}} \right\} \right)$.
\end{lemma}

\begin{proof}
\ Let $a_{n} - \frac{ \lambda }{ n^{1 - k\alpha} \ln d } = b_{n} = 1 - \eta_{1}$. Then
\begin{align}
\nonumber
     &\psi(\eta_{1})^{n}\left(\frac{1}{d}\right)^{n\eta_{1}}\left(1-\frac{1}{d}\right)^{n-n\eta_{1}}\\
\nonumber
=    &\left( 1 - \frac{ \lambda }{ n^{1 - k\alpha}a_{n} \ln d } \right)^{ n \eta_{1} } \left( 1 + \frac{ \lambda }{ n^{ 1 - (k-1)\alpha } b_{n} \ln d } \right)^{n}\\
\nonumber
=    &\exp\bigg\{ n\eta_{1} \ln \left( 1 - \frac{ \lambda }{ n^{ 1 - k\alpha } a_{n} \ln d } \right) + n\ln \left( 1 + \frac{ \lambda }{ n^{ 1 - (k-1)\alpha } b_{n} \ln d } \right) \bigg\},
\end{align}
where
\begin{align}
\nonumber
&n\eta_{1}\ln \left( 1 - \frac{ \lambda }{ n^{ 1 - k\alpha } a_{n} \ln d } \right) + n \ln \left( 1+\frac{ \lambda }{ n^{ 1 - (k-1)\alpha } b_{n} \ln d } \right)\\
\nonumber
\leq               &-n \eta_{1} \left( \frac{ \lambda }{ n^{ 1 - k\alpha } a_{n} \ln d } + \frac{ \lambda^{2} }{ 2n^{ 2(1 - k\alpha) } a_{n}^{2} (\ln d)^{2} } \right) + \frac{ n\lambda }{ n^{ 1 - (k-1)\alpha } b_{n} \ln d }\\
\nonumber
\leq               &\left( \frac{ \lambda }{ b_{n} } - \frac{ \lambda }{ a_{n} } \right) \frac{ n^{ ( k - 1 )\alpha } }{ \ln d } - \left( \frac{ \lambda^{2} }{ a_{n} } + \frac{ \lambda^{2} }{ 2a_{n}^{2} } \right) \frac{ n^{ ( 2k - 1 )\alpha - 1 } }{(\ln d)^{2}}\\
\nonumber
=                  & - \frac{ \lambda^{2} }{2}\left( \frac{2}{a_{n}}+\frac{1}{a_{n}^{2}}-\frac{2}{a_{n} b_{n}} \right) \frac{ n^{(2k-1)\alpha-1} }{ (\ln d)^{2} }
= - \left( \frac{ \lambda^{2} }{2} + o(1) \right)\frac{ n^{(2k-1)\alpha-1} }{ (\ln d)^{2} }.
\end{align}
Thus, the Lemma follows immediately from Lemmas \ref{lemma42} and \ref{lemma45}.
\end{proof}

\begin{lemma}\label{lemma47}
\ $\sum_{n\eta_{1}}^{n \eta_{2}}\Phi (S)=o(1)$.

\end{lemma}
\begin{proof}\ By Lemmas \ref{lemma44} and \ref{lemma46}, the proof  is straightforward.
\end{proof}

\begin{lemma}\label{lemma48}
Let $k \geq \frac{\tau\ln \tau}{\tau - 1}$ and let $\varphi(s) = r\ln \left( 1 + \frac{p}{1-p} s^{k} \right) - s$,
then $\Phi( ns ) \le \exp\{ -( \beta + o(1) ) m  \}$ \ holds uniformly for all $\eta_{2} \leq s \leq \eta_{3}$,
where $- \beta = \max_{\eta_{2} \leq s \leq \eta_{3}} \varphi(s) < 0$.

\end{lemma}

\begin{proof} Note that ${n \choose ns} \leq {n \choose n/2} = \exp\{o( m )\}$. Then
\begin{align}
\Phi( ns )\leq \exp\left\{\left( \varphi(s) + o(1) \right)m \right\} \label{27}
\end{align}
holds uniformly for all $\eta_{2} \leq s \leq \eta_{3}$.

We can rewrite $r_{cr} \ln \left(1 + \frac{p}{1 - p} s^{k}\right) - s < 0$ as $u(s) > 0$ for $s\in(0, 1)$,
where $u(s) = \tau^{s} - (\tau - 1) s^{k} - 1$.

Now we consider the equation $u^{'}(s) = \tau^{s} \ln\tau - k(\tau - 1) s^{k - 1} = 0$.

Note that $u(0) = u(1) = 0$, then by Rolle's Theorem, the equation $u^{'}(z) = 0$ has at least one solution in $(0, 1)$.

On the other hand, we can rewrite the equation $u^{'}(s) = 0$ as
\begin{align}
\nonumber
\frac{\ln \tau}{k - 1}s + \frac{\ln\ln \tau - \ln k - \ln (\tau - 1)}{k - 1} = \ln s.
\end{align}
In rectangular plane coordinate system, the left side of the equation is a straight line, and the right side is a concave curve,
thus they have at most two points of intersection.

Note that $\tau > 1$, so $u^{'}(0) = \ln\tau > 0$ and $\lim_{s\rightarrow + \infty}u^{'}(s) = +\infty$.
Then all the possible cases of $u^{'}(s)$ in $(0,1)$ are as follows.

$(1)$\ \ \ $u^{'}(s) = 0$ has one solution $s_{0}$ in $(0, 1)$
\begin{center}
\begin{tabular}{ccc}
$s$               &$(0,\  s_{0})$      &$(s_{0},\  1)$\\
\hline
$u^{'}$               &$>0$         &$<0$
\end{tabular}
\end{center}

$(2)$\ \ \ $u^{'}(s) = 0$ has two solutions $s_{1}, s_{2}$ $(s_{1} < s_{2})$ in $(0,1)$
\begin{center}
\begin{tabular}{cccc}
$s$               &$(0,\  s_{1})$      &$(s_{1},\  s_{2})$         &$(s_{2},\  1)$\\
\hline
$u^{'}$               &$>0$         &$<0$                &$>0$
\end{tabular}
\end{center}
Thus, $u(s) > 0$ in $(0, 1)$ is equivalent to $u^{'}(1) \leq 0$, i.e., $k \geq \frac{\tau\ln \tau}{\tau - 1}$.
Then $- \beta = \max_{\eta_{2} \leq s \leq \eta_{3}} \varphi(s) < 0$,
and now the proof of the Lemma follows by ($\ref{27}$).
\end{proof}

\begin{lemma}\label{lemma49}
 $\sum_{S=n \eta_{2}}^{n \eta_{3}}\Phi (S)=o(1)$.

\end{lemma}

\begin{proof}\ By Lemma \ref{lemma48}, the proof of the Lemma is straightforward.
\end{proof}

\begin{lemma}\label{lemma410}
$\sum_{S=n \eta_{3}}^{n}\Phi(S)=o(1)$.
\end{lemma}

\begin{proof}\ Note that ${n \choose S} \leq {n \choose n/2} = \exp\{ o( m ) \}$. Then
\begin{align}
\nonumber
\Phi(S)\leq \exp\{- (r \ln(1-p) + \eta_{3} + o(1)) m \}
\end{align}
hold uniformly for all $\eta_{3} \leq s \leq 1$.

Note that $r\ln(1-p) + \eta_{3} > 0$, thus the proof  is finished.
\end{proof}

\section{Proof of Theorems}
{\flushleft\emph{Proof of Theorem \ref{th31}}.}\ Note that $\mathbb{E}[X] = o(1)$, provided that $r > r_{cr}$.
Thus, by the Markov inequality $\text{Pr}[X>0]\leq \mathbb{E}[X]$, ($\ref{3}$) is proved.\qed

{\flushleft
The proof of (\ref{a}) follows from Lemmas 4.3, 4.7, 4.9 and 4.10.

The proof of Theorem 3.2 can be obtained similarly.
}

{\flushleft {\bf Remark 5.1}.}\ We have fully exploited the power of the second moment method in Model RB: In Theorems 3.1 and 3.2, if $(2k-1)\alpha\leq 1$, or $k<\frac{\tau\ln \tau}{\tau-1}$, the second moment method will fail to give nontrivial results.

{\flushleft {\bf Claim 5.2}.}\ $\mathbb{E}[X]^{2} / \mathbb{E}[X^{2}] = o(1)$, \emph{provided that} $(2k-1)\alpha\leq 1$.

{\flushleft {\bf Claim 5.3}.}\ \emph{If} $r$ \emph{is close enough to} $r_{cr}$ $(r < r_{cr})$, \emph{then} $\mathbb{E}[X]^{2} / \mathbb{E}[X^{2}] = o(1)$, \emph{provided that} $k<\frac{\tau\ln \tau}{\tau-1}$.

Let $S_{n}=X_{1}+X_{2}+\cdots+X_{n}$ is the sum of $n$ independent random variables, where, for each $i$,
$\text{Pr}[X_{i}=1] = \text{p}$ and $\text{Pr}[X_{i}=0] = 1 - \text{p}$.
Then $\text{Pr}[S_{n}=i] = {n \choose i} \text{p}^{i} (1 - \text{p} )^{n-i}$, for $i=0, 1, 2, \cdots, n$.

The following theorem can be viewed as a special case of the Central Limit Theorem, and especially when p is a constant, it is
sometimes called the DeMoivre-Laplace Theorem \cite{a5, DeMoivre}.
{\flushleft {\bf Theorem 5.4}}.
\emph{The binomial distribution} $\text{Binomial}(n,\text{p})$ \emph{for} $S_{n}$, \emph{satisfies}, \emph{for two constants} $a$ \emph{and} $b$,
\begin{align}
\nonumber
\lim_{n\rightarrow\infty}\text{Pr}[a\sigma<S_{n} - n\text{p}<b\sigma]=\frac{1}{\sqrt{2\pi}}\int_{a}^{b}e^{- x^{2}/2}dx
\end{align}
\emph{where} $\sigma=\sqrt{n\text{p}(1 - \text{p})}$, \emph{provided that} $n\text{p}(1 - \text{p})\rightarrow \infty$ \emph{as} $n\rightarrow \infty$.

{\flushleft\emph{Proof of Claim 5.2}}.\ Let $\text{p} = 1/d$ and let $0 < \rho < \alpha$ be a constant. Note that $(2k-1)\alpha \leq 1$, then
\begin{align}
\nonumber
\lim_{n\rightarrow\infty}\frac{ n\eta_{1} - n\text{p} }{\sigma} = 0,\ \text{and} \lim_{n\rightarrow\infty}\frac{ n^{1 - \rho} - n\text{p} }{\sigma} = \infty.
\end{align}
Thus, by Theorem 5.4, we have
\begin{align}
\sum_{S=n\eta_{1}}^{n^{1 - \rho}}\text{B}(S)=\frac{1}{2} + o(1). \label{f3}
\end{align}

Note that $\left( f(s) + \frac{ p g(s) }{ (1-p) \left(1-d^{-k}\right) n} \right)^{ r m } = f(s)^{ r m } + o(1)$ holds uniformly for all $s \leq n^{ -\rho}$. By Lemmas 4.1, 4.2, and (\ref{f3}), we have
\begin{align}
\nonumber
\frac{\mathbb{E}[X^{2}]}{\mathbb{E}[X]^{2}}\geq \Big( 1 + o(1) \Big)\sum_{S=n\eta_{1}}^{n^{1 - \rho}}\text{B}(S)\text{W}(S)\geq \left(\frac{1}{2}+o(1)\right)\exp\left\{\frac{ k\lambda p r }{1-p}\right\}.
\end{align}
Let $\lambda\rightarrow\infty$, then $\mathbb{E}[X^{2}] / \mathbb{E}[X]^{2}\rightarrow \infty$, and the proof is finished.   \qed

{\flushleft\emph{Proof of Claim 5.3}}.\ Let $k<\frac{\tau\ln\tau}{\tau-1}$, then there  exists a constant $\theta\in (0, 1)$ such that $r_{cr}\ln(1 + \frac{p}{1-p}\theta^{k}) - \theta > 0$.

Note that $r$ is close enough to $r_{cr}$ $(r < r_{cr})$, thus $r\ln\left( 1 + \frac{p}{1-p}\theta^{k} \right) - \theta > 0$, then
\begin{align}
\nonumber
                          &\text{B}(n\theta) \left( f(\theta) + \frac{ p g(\theta) }{ (1-p) \left(1-d^{-k}\right) n} \right)^{ r m } \\
\nonumber
\geq                      &\left(\frac{1}{d}\right)^{n\theta}\left(1+\frac{p}{1-p}\theta^{k} + o(1) \right)^{ r m }\\
\nonumber
=                         &\exp\left\{\left( r\ln\left( 1 + \frac{p}{1-p}\theta^{k} \right) - \theta + o(1) \right) m \right\}.
\end{align}
Thus, the proof of the Claim is finished. \qed

\section{Conclusions}

In this paper, we exploit the power of the second moment method in proving exact phase transitions of Model RB, and introduce some useful techniques to analyse the limit behavior of the second moment. Our results are the best by using the second moment method in that, it fails to give nontrivial results once the conditions are further relaxed (see Claim 5.2), thus new tools other than the second moment method will be needed to entail better results. To be precise, we show that  the requirement of domain size $d=n^{\alpha}$ in Model RB can be relaxed from $\alpha>\frac1k$ to $\alpha>\frac{1}{2k-1}$. In this way, benchmarks based on the exact phase transitions of model RB can be generated more easily for algorithm research.

For further work, one interesting open problem is whether this requirement can be further relaxed (i.e., $d$ grows at a lower speed)
while exact phase transitions can still be guaranteed.


\begin{thebibliography}{}




\bibitem{c2} D. Achlioptas, L.M. Kirousis, E. Kranakis, D. Krizanc, M.S.O. Molloy and Y.C. Stamatiou, Random Constraint Satisfaction: A More Accurate Picture, Constraints 6, 329-344, 2001.




\bibitem{a5} F. Chung and L. Lu, Complex Graphs and Networks, AMS, 2006.




\bibitem{DeMoivre} W. Feller, Martingales, An Introduction to Probability Theory and its Applications, Vol. 2. New York: Wiley, 1971.

\bibitem{FS2011} Y. Fan, J. Shen, On the phase transitions of random
$k$-constraint satisfaction problems, Artificial Intelligence, 175,
 914-927, 2011.

\bibitem{FSX2012}Y. Fan, J. Shen and K. Xu, A general model and thresholds for random
constraint satisfaction problems, Artificial Intelligence, 193, 1-17,
2012.



\bibitem{c7} A. Frieze and M. Molloy, The satisfiability threshold for randomly generated binary constraint satisfaction problems, Random Struct. Algorithms 28, 323-339, 2006.

\bibitem{Frieze2005} A. Frieze and N.C. Wormald, Random $k$-SAT: A tight threshold for moderately growing $k$, Combinatorica 25, 297-305, 2005.




\bibitem{e1}V. Guruswami, R. Manokaran, and P. Raghavendra, Beating the random ordering is hard: Inapproximability of maximum acyclic subgraph, in: 49th Annual IEEE Symposium on Foundations of Computer Science,  573-582, 2008.







\bibitem{d3} L. Li, T. Liu and K. Xu, Exact phase transition of backtrack-free search with implications on the power of greedy algorithms, arXiv: 0811.3055, 2008.

\bibitem{liu2012} J. Liu, Z. Gao and K. Xu, A Note on Random $k$-SAT for Moderately Growing $k$, E-JC 19, \# P24, 2012.








\bibitem{smith1996} B.M. Smith and M.E. Dyer, Locating the phase transition in binary constraint satisfaction problems, Artif. Intell. 81, 155-181, 1996.

\bibitem{xu19}
K. Xu, F. Boussemart, F. Hemery and C. Lecoutre, A simple model to generate hard satisfiable instances. 19th International Joint Conferences on Artificial Intelligence (IJCAI), 337-342.

\bibitem{xu2000}
K. Xu and W. Li, Exact phase transitions in random constraint satisfaction problems. Journal of Artificial Intelligence Research, 12(1): 93-103, 2000.


\bibitem{xu2006}
K. Xu, W. Li, Many hard examples in exact phase transitions. Theoretical Computer Science 355, 291-302, 2006.

\bibitem{xu2007} K. Xu and W. Li, Random Constraint Satisfaction: Easy Generation of Hard (satisfiable) Instances. Artificial Intelligence, 171 (8-9): 514-534, 2007.

\bibitem{xuwei2015} W. Xu, P. Zhang, T. Liu and F. Gong, The solution space structure of random constraint satisfaction problems with growing domains. Journal of Statistical Mechanics: Theory and Experiment, P12006, 2015.


\bibitem{zhao2012}C. Zhao, P. Zhang, Z. Zheng, K. Xu, Analytical and belief-propagation studies of random constraint satisfaction problems with growing domains. Physical Review E 85, 016106, 2012.

\bibitem{a15} R.M. Young, Euler's Constant, Math. Gaz. 75, 187-190, 1991.
\end{thebibliography}
\end{document}